\documentclass[12pt,a4paper,reqno,twoside]{amsart}

\usepackage[english]{babel}
\usepackage{stmaryrd}
\usepackage{dsfont}
\usepackage[symbol*,ragged]{footmisc}
\usepackage[colorlinks,linkcolor=red,anchorcolor=blue,citecolor=blue,urlcolor=blue]{hyperref}
\usepackage{color,xcolor}

\usepackage{geometry}
\usepackage{amssymb}
\usepackage{amsmath}
\usepackage{mathrsfs}
\usepackage{amsfonts}
\usepackage{epsfig}

\usepackage{amsthm}
\usepackage{amsxtra}
\usepackage{bbding}
\usepackage{epsfig}
\usepackage{graphicx}
\usepackage{latexsym}
\usepackage{mathbbol}
\usepackage{bbold}

\usepackage{pifont}
\usepackage{wasysym}
\usepackage{skull}
\usepackage{float}

\DeclareSymbolFontAlphabet{\mathbb}{AMSb}
\DeclareSymbolFontAlphabet{\mathbbol}{bbold}

\usepackage{amscd}
\usepackage[all]{xy}
\allowdisplaybreaks[4]
\usepackage{setspace}

\theoremstyle{plain}
\newtheorem{theorem}{\normalfont\scshape Theorem}[section]
\newtheorem{proposition}{\normalfont\scshape Proposition}[section]
\newtheorem{lemma}[proposition]{\normalfont\scshape Lemma}

\newtheorem*{corollary*}{\normalfont\scshape Corollary}

\theoremstyle{remark}
\newtheorem*{remark*}{\normalfont\scshape Remark}
\newtheorem*{notation}{\normalfont\scshape Notation}

\numberwithin{equation}{section}
\addtocounter{footnote}{1}

\renewcommand{\footnoterule}{
  \kern -3pt
  \hrule width 2.5in height 0.4pt
  \kern 3pt
}

\makeatletter
\@ifundefined{MakeUppercase}{}{}
\makeatother

\begin{document}
	
\title[On the quadratic Waring--Goldbach problem with primes in Piatetski--Shapiro sets]
	  {On the quadratic Waring--Goldbach problem with primes in Piatetski--Shapiro sets}

\author[Meng Gao, Jinjiang Li, Linji Long, Min Zhang]
       {Meng Gao \quad \& \quad Jinjiang Li \quad \& \quad Linji Long \quad \& \quad Min Zhang}

\address{[Meng Gao] Department of Mathematics, China University of Mining and Technology,
         Beijing 100083, People's Republic of China}

\email{\textcolor{black}{meng.gao.math@gmail.com}}

\address{[Jinjiang Li] (Corresponding author) Department of Mathematics, China University of Mining and Technology,
         Beijing 100083, People's Republic of China}

\email{\textcolor{black}{jinjiang.li.math@gmail.com}}

\address{[Linji Long] Department of Mathematics, China University of Mining and Technology,
         Beijing 100083, People's Republic of China}

\email{\textcolor{black}{linji.long.math@gmail.com}}

\address{[Min Zhang] School of Applied Science, Beijing Information Science and Technology University,
		 Beijing 100192, People's Republic of China  }

\email{\textcolor{black}{min.zhang.math@gmail.com}}

\date{}

\footnotetext[1]{Jinjiang Li is the corresponding author. 

\quad\,\,
{\textbf{MR(2020) Subject Classification}}: 11P05, 11P32, 11D85

}

\begin{abstract}
In this paper, it is proved that, for any $\gamma_1,\gamma_2,\gamma_3,\gamma_4,\gamma_5\in(\frac{28}{29},1)$, every sufficiently large integer $n$ subject to $n\equiv5\pmod{24}$ can be represented as the sum of five squares of primes, i.e.,
\begin{equation*}
n=p_1^2+p_2^2+p_3^2+p_4^2+p_5^2,
\end{equation*}
such that $p_i=\lfloor m_i^{1/\gamma_i}\rfloor$	for some $m_i\in\mathbb{N}^+$ for each $1\leqslant i\leqslant 5$.
This result constitutes an improvement upon the previous result of Zhang and Zhai \cite{Zhang-Zhai-2005}.
\end{abstract}

\maketitle

\section{Introduction and main result}
For fixed integer $k\geqslant1$ and sufficiently large integer $n$, the well--known Waring--Goldbach problem is devoted to investigating the solvability of the following Diophantine equality
\begin{equation}\label{Waring-Goldbach-gene}
n=p_1^k+p_2^k+\dots+p_s^k
\end{equation}
in prime variables $p_1,p_2,\dots,p_s$. A formal application of the Hardy--Littlewood method suggests that whenever $s$ and $k$ are natural numbers with $s\geqslant k+1$, then (\ref{Waring-Goldbach-gene}) holds with the expected asymptotic formula
\begin{equation*}
\sum_{n=p_1^k+\dots+p_s^k}1=\mathfrak{S}_{k,s}(n)\frac{\Gamma^s(1+1/k)}{\Gamma(s/k)}\frac{n^{s/k-1}}{(\log n)^s}
+O\bigg(\frac{n^{s/k-1}\log\log n}{(\log n)^{s+1}}\bigg),
\end{equation*}
where $\mathfrak{S}_{k,s}(n)$ is the singular series. Denote by $H(k)$ the least integer $s$ such that every sufficiently large positive integer satisfying some congruent conditions may be written as in the shape of (\ref{Waring-Goldbach-gene}) with $p_1,\dots,p_s$  prime numbers. The first general bound for $H(k)$ was obtained by Hua \cite{Hua-1938}, who showed that
\begin{equation}\label{H(k)-upper}
H(k)\leqslant2^k+1
\end{equation}
for every $k\geqslant1$. This result, which generalizes Vinogradov's celebrated three primes theorem \cite{Vinogradov-1937}, remains the best known bound on $H(k)$ for $k=1,2,3$. When $k\geqslant4$, on the other hand, the bound (\ref{H(k)-upper}) has been sharpened considerably.

Let $\gamma\in(\frac{1}{2},1)$ be a fixed real number. The Piatetski--Shapiro sequences are sequences of the form
\begin{equation*}
 \mathscr{N}_{\gamma}:=\big\{\lfloor n^{1/\gamma}\rfloor:\,n\in \mathbb{N}^+\big\}.
\end{equation*}
Such sequences have been named in honor of Piatetski--Shapiro, who \cite{Piatetski-Shapiro-1953}, in
1953, proved that $\mathscr{N}_{\gamma}$ contains infinitely many primes provided that $\gamma\in(\frac{11}{12},1)$. The prime numbers of the form $p=\lfloor n^{1/\gamma}\rfloor$ are called \textit{Piatetski--Shapiro primes of type $\gamma$}. More precisely, for such $\gamma$ Piatetski--Shapiro \cite{Piatetski-Shapiro-1953} showed that the counting function
\begin{equation*}
 \pi_\gamma(x):=\#\big\{\textrm{prime}\,\, p\leqslant x:\,p=\lfloor n^{1/\gamma}\rfloor\,\,\textrm{for some}\,\,
 n\in\mathbb{N}^+ \big\}
\end{equation*}
satisfies the asymptotic property
\begin{equation*}
\pi_{\gamma}(x)=\frac{x^{\gamma}}{\log x}(1+o(1))
\end{equation*}
as $x\to\infty$. Since then, the range for $\gamma$ of the above asymptotic formula in which it is known that $\mathscr{N}_{\gamma}$ contains infinitely many primes has been enlarged many times (e.g., see the literatures \cite{Kolesnik-1967,Leitmann-1975,Leitmann-1980,Heath-Brown-1983,Kolesnik-1985,Liu-Rivat-1992,
Rivat-1992,Rivat-Sargos-2001}) over the years and is currently known to hold for all $\gamma\in(\frac{2426}{2817},1)$ thanks to Rivat and Sargos \cite{Rivat-Sargos-2001}. Rivat and Wu \cite{Rivat-Wu-2001} also showed that there exist infinitely many Piatetski--Shapiro primes for $\gamma\in(\frac{205}{243},1)$ by showing a lower bound of $\pi_\gamma(x)$ with the expected order of magnitude. We remark that if $\gamma>1$ then $\mathscr{N}_\gamma$ contains all natural numbers, and hence all primes, particularly.

Based on the previous result, it is natural to investigate the Waring--Goldbach problem with prime variables restricted to Piatetski--Shapiro set. For the linear case, in 1992, Balog and Friedlander \cite{Balog-Friedlander-1992} firstly found an asymptotic formula for the number of solutions of the equation (\ref{Waring-Goldbach-gene}) with three variables restricted to the Piatetski--Shapiro primes. An interesting corollary of their theorem is that every sufficiently large odd integer can be written as the sum of three Piatetski--Shapiro primes of type $\gamma$, provided that $\gamma\in(\frac{20}{21},1)$. Afterwards, their studies in this direction were subsequently continued by Jia \cite{Jia-1995} and by Kumchev \cite{Kumchev-1997}, and generalized by Cui \cite{Cui-2004} and Li and Zhang \cite{Li-Zhang-2018}, consecutively and respectively.
In 2025, Sun, Du and Pan \cite{Sun-Du-Pan-2025} used the transference principle to improve the hybrid problem of Vinogradov's three prime theorem in which each prime variable is constrained into Piatetski--Shapiro primes. To be specific, they proved that, for any $\gamma_1,\gamma_2,\gamma_3\in(\frac{35}{41},1)$, every sufficiently large odd integer $n$ can be represented as the sum of three primes, i.e., $n=p_1+p_2+p_3$, where each $p_i$ is of the form
$p_i=\lfloor n_i^{1/\gamma_{i}}\rfloor$ for some integers $n_i\in\mathbb{N}^+$ ($i=1,2,3$).

In 1998, Zhai \cite{Zhai-1998} considered the hybrid problem of quadratic Waring--Goldbach problem with each prime variable restricted to Piatetski--Shapiro sets. To be specific, he proved that, for $\gamma\in(\frac{43}{44},1)$ fixed, every sufficiently large integer $n$ satisfying $n\equiv5\pmod{24}$ can be written as the sum of five squares of primes with each prime of the form $\lfloor m^{1/\gamma}\rfloor$. Later, in 2005, Zhang and Zhai \cite{Zhang-Zhai-2005} improved the result of Zhai \cite{Zhai-1998} and enlarge the range to $\frac{249}{256} <\gamma<1$. Motivated by the work of Sun, Du and Pan \cite{Sun-Du-Pan-2025}, we shall utilize the transference principle to investigate the quadratic case of (\ref{Waring-Goldbach-gene}) with Piatetski--Shapiro primes.

In this paper, we shall improve the result of Zhang and Zhai \cite{Zhang-Zhai-2005}, and establish the following theorem.
\begin{theorem}\label{Theorem-1}
For any $\gamma_1,\gamma_2,\gamma_3,\gamma_4,\gamma_5\in(\frac{28}{29},1)$, every sufficiently large positive integer $n$ subject to $n\equiv5\pmod{24}$ can be represented as the sum of five squares of primes, i.e.,
\begin{equation*}
n=p_1^2+p_2^2+p_3^2+p_4^2+p_5^2,
\end{equation*}
each of which is of the form $p_i=\lfloor m_i^{1/\gamma_i}\rfloor,\,i=1,2,3,4,5$.
\end{theorem}

\begin{remark*}
In order to compare our result with the previous results of Zhang and Zhai \cite{Zhang-Zhai-2005},
we list the numerical results as follows:
\begin{equation*}
\frac{249}{256}=0.97265625,\qquad \frac{28}{29}=0.96551724\dots.
\end{equation*}
\end{remark*}

\begin{notation}
For any positive integer $N$, let $[N]:=\{1,\ldots,N\}$. We write $\mathbb{T}$ for $\mathbb{R}/\mathbb{Z}$.
For a set $\mathcal{A}$, we write $\mathds{1}_{\mathcal{A}}(x)$ for its characteristic function. The letter $p$, with or without subscript, always denotes a prime number. Denote by $\mathcal{P}$ the set of all the primes. As usual, we use $\varphi(n)$ to denote Euler's function, and $\lfloor x\rfloor$ to denote the integral part of $x\in\mathbb{R}$. Also, we write $e(t)=\exp(2\pi it)$. $f(x)\ll g(x)$ means that $f(x)=O(g(x))$; $f(x)\asymp g(x)$ means that $f(x)\ll g(x)\ll f(x)$.
Let $f:\mathcal{B}\rightarrow\mathbb{C}$ be a function and $\mathcal{B}_1$ be a non--empty finite subset of $\mathcal{B}$. Denote by $\mathbb{E}_{x\in\mathcal{B}_1}f(x)$ the average value of $f$ on $\mathcal{B}_1$, i.e.,
\begin{equation*}
\mathbb{E}_{x\in\mathcal{B}_1}f(x)=\dfrac{1}{|\mathcal{B}_1|}\sum\limits_{x\in\mathcal{B}_1}f(x).
\end{equation*}
	For finitely supported functions $ f,g:\mathbb{Z}\rightarrow \mathbb{C} $ , we define the convolution
$f\ast g$ by
\begin{equation*}
f\ast g\,(n)=\sum\limits_{a+b=n}f(a)g(b).
\end{equation*}
For a finitely supported function $f:\mathbb{Z}\rightarrow\mathbb{C}$, then the Fourier transform of $f$ is defined as
\begin{equation*}
\widehat{f}(\alpha)=\sum\limits_{n\in\mathbb{Z}}f(n)e(n\alpha),\qquad \alpha\in\mathbb{T}.
\end{equation*}
The $L^\infty$--norm and $L^r$--norm of function $f:\mathbb{T}\rightarrow\mathbb{C}$ are defined respectively as
\begin{equation*}
\big\|f\big\|_\infty=\sup_{\alpha}\big|f(\alpha)\big|,\qquad
\big\|f\big\|_{r}=\bigg(\int_{\mathbb{T}}\big|f(\alpha)\big|^{r}\mathrm{d}\alpha\bigg)^{1/r}.
\end{equation*}
\end{notation}

\section{Transference principle}

In order to prove Theorem \ref{Theorem-1}, we need to utilize the following transference principle, which is established by Salmensuu \cite{Salmensuu-2020}.
\begin{proposition}\label{transference principle}
Let $s\geqslant3$ and $\varepsilon,\eta\in(0,1)$. Suppose that $N$ is a natural number. For each
$i\in\{1,\dots,s\}$, let $f_i:[N]\to\mathbb{R}_{\geqslant0}$ be a function which satisfies the following postulations.

\noindent
(1) (Mean value condition) For each arithmetic progression $P\subseteq [N]$ with $|P|\geqslant\eta N$, we have
\begin{equation*}		
\mathbb{E}_{n\in P}f_i(n)\geqslant1/s+\varepsilon;
\end{equation*}	
	
\noindent		
(2) (Pseudorandomness condition) There exists a majorant $\nu_i:[N]\to\mathbb{R}_{\geqslant0}$ with
$f_i\leqslant\nu_i$ pointwise, such that $\|\widehat{\nu_i}-\widehat{\mathds{1}_{[N]}}\|_\infty\leqslant\eta N$;

\noindent	
(3) (Restriction estimate) We have $\|\widehat{f_i}\|_q\leqslant KN^{1-1/q}$ for some $q\in(s-1,a)$ and
$K\geqslant 1$.

\noindent	
Then, for each $n\in[N/2,N]$, there holds
\begin{equation*}			
f_1*\dots*f_s(n)\geqslant\big(c(\varepsilon)-O_{\varepsilon,K,q}(\eta)\big)N^{s-1},
\end{equation*}			
where $c(\varepsilon)>0$ is a constant depending only on $\varepsilon$.
\end{proposition}
\begin{proof}
See Lemma 1 of Salmensuu \cite{Salmensuu-2020}.
\end{proof}

\begin{lemma}\label{trans}
Let $\theta\in(0,1)$ and $\gamma\in(\frac{75}{82},1)$. Let $\delta^*$ be any positive constant satisfying
\begin{equation*}
82(1-\gamma)+87\delta^*<7.
\end{equation*}	
Then, for $0<\delta<\delta_0=\min\{\gamma-\frac{1}{2},\frac{\delta^*}{3}\}$, there holds
\begin{equation*}
\frac{1}{\gamma}\sum_{\substack{p\leqslant x\\ p\in\mathscr{N}_{\gamma}}}p^{2-\gamma}(\log p)\cdot e(\theta p^2)
=\sum_{p\leqslant x}p(\log p)\cdot e(\theta p^2)+O(x^{2-\delta-\varepsilon}),
\end{equation*}
where $\varepsilon>0$ is arbitrarily small.
\end{lemma}
\begin{proof}
See Lemma 4.5 of Ren, Zhang and zhang \cite{Ren-Zhang-Zhang-2024}.
\end{proof}
For $m\in\mathbb{N}^+$, define
\begin{equation*}
Z(m):=\big\{b\in\mathbb{Z}_{m}:\,\exists\, t\in\mathbb{Z}_{m},\,(t,m)=1,\,t^2\equiv b\!\!\!\!\pmod{m}\big\} ,
\end{equation*}
where $\mathbb{Z}_{m}:=\mathbb{Z}/m\mathbb{Z}$. For any $b\in Z(m)$, define
\begin{equation*}
\sigma_{m}(b):=\#\big\{z\in[m]:\,z^2\equiv b\!\!\!\!\pmod{m}\big\} .
\end{equation*}
Let $n_{0}$ be a sufficiently large positive integer and $w:=\log\log\log n_{0}$. Let
\begin{equation*}
W:=8\prod\limits_{2<p\leqslant w}p,
\end{equation*}
and
\begin{equation*}
N:=\lfloor n_{0}/W\rfloor +1.
\end{equation*}
For any $\gamma\in(1/2,1)$ and $b\in[W]$ with $b\in Z(W)$, define function
$f_{b,\gamma}:[N]\rightarrow \mathbb{R}_{\geqslant0}$ by
\begin{equation*}
f_{b,\gamma}(n) :=
 \begin{cases}
 \displaystyle\frac{2\varphi(W)}{W\sigma_{W}(b)}\gamma^{-1}p^{2-\gamma}\log p,
      & \textrm{if\, $Wn+b=p^2$ and $p\in\mathscr{N}_{\gamma}\cap\mathcal{P}$}, \\
	\quad 0, & \text{otherwise}.
		\end{cases}
\end{equation*}
\begin{lemma}\label{pseudorandomness}
Suppose that $\gamma\in(\frac{75}{82},1)$. For $b\in[W]$ with $b\in Z(W)$, one has
\begin{equation*}
\big\|\widehat{f_{b,\gamma}}-\widehat{1_{[N]}} \big\|_\infty=o(N).
\end{equation*}
\end{lemma}
\begin{proof}
Define function $\lambda_{b}:\,[N]\rightarrow \mathbb{R}_{\geqslant0}$ by
\begin{equation*}
\lambda_{b}(n) :=
 \begin{cases}
  \displaystyle\frac{2\varphi(W)}{W\sigma_{W}(b)}p\log p,
    & \textrm{if\,$Wn+b=p^2$ and $p\in\mathcal{P}$}, \\
    \quad 0, & \text{otherwise}.
 \end{cases}
\end{equation*}
By the definition of $f_{b,\gamma}$, we have
\begin{align}\label{Fourier transform of f}
         \widehat{f_{b,\gamma}}(\alpha)
= & \,\, \dfrac{2\varphi(W)}{W\sigma_{W}(b)}\gamma^{-1}
         \sum_{\substack{b<p^{2}\leqslant WN+b\\p^2\equiv b \!\!\!\!\!\pmod{W}\\ p\in\mathscr{N}_\gamma}}
         p^{2-\gamma}(\log p)\cdot e\bigg(\frac{p^{2}-b}{W}\cdot\alpha\bigg)
                \nonumber \\
= & \,\, \dfrac{2\varphi(W)}{W\sigma_{W}(b)}\gamma^{-1}
         \sum_{\substack{p\leqslant\sqrt{WN+b}\\ p^2\equiv b \!\!\!\!\!\pmod{W}\\ p\in\mathscr{N}_\gamma}}
         p^{2-\gamma}(\log p)\cdot e\bigg(\frac{p^{2}-b}{W}\cdot\alpha\bigg)+O(W).
\end{align}
Trivially, one has
\begin{equation*}
\frac{1}{W}\sum_{1\leqslant j\leqslant W}e\bigg(j\cdot\frac{p^{2}-b}{W}\bigg)=
\begin{cases}
1, & \textrm{if\,\,$p^2 \equiv b \!\!\!\!\pmod{W}$} , \\
0, & \textrm{otherwise},
\end{cases}
\end{equation*}
which combined with \eqref{Fourier transform of f} yields that
\begin{equation}\label{Fourier transform of f (1)}
\widehat{f_{b,\gamma}}(\alpha)=\frac{2\varphi(W)}{W^{2}\sigma_{W}(b)}
\sum_{1\leqslant j\leqslant W}e\bigg(-\frac{b}{W}(\alpha+j)\bigg)\gamma^{-1}
\sum_{\substack{p\leqslant\sqrt{WN+b}\\ p\in\mathscr{N}_\gamma}}p^{2-\gamma}(\log p)\cdot
e\bigg(p^2\cdot\frac{\alpha+j}{W}\bigg)+O(W).
\end{equation}
Similarly, we have
\begin{equation}\label{Fourier transform of lambda}
\widehat{\lambda_{b}}(\alpha)=\frac{2\varphi(W)}{W^{2}\sigma_{W}(b)}
\sum_{1\leqslant j\leqslant W}e\bigg( -\frac{b}{W}(\alpha+j)\bigg)
\sum_{p\leqslant\sqrt{WN+b}}p(\log p)\cdot e\bigg(p^2\cdot\frac{\alpha+j}{W}\bigg)+O(W).
\end{equation}
By (\ref{Fourier transform of f (1)}), (\ref{Fourier transform of lambda}) and Lemma \ref{trans}, we have
\begin{align*}
           \big\|\widehat{f_{b,\gamma}}-\widehat{\lambda_{b}}\big\|_{\infty}
\ll & \,\, W+\sup_{\alpha\in\mathbb{T}}\frac{\varphi(W)}{W^{2}\sigma_{W}(b)}\sum_{1\leqslant j\leqslant W}
           \Bigg|\frac{1}{\gamma}\sum_{\substack{p\leqslant\sqrt{WN+b}\\ p\in \mathscr{N}_\gamma}}
           p^{2-\gamma}(\log p)\cdot e\bigg(p^2\cdot\frac{\alpha+j}{W}\bigg)
				 \nonumber \\
    & \,\, -\sum_{p\leqslant\sqrt{WN+b}}p(\log p)\cdot e\bigg(p^2\cdot\frac{\alpha+j}{W}\bigg)\Bigg|
                 \nonumber \\
\ll & \,\, N^{1-\varepsilon}.
\end{align*}
Following the arguments, i.e., Section 3 and Section 4 of Chow \cite{Chow-2018}, or Section 5 of Gao \cite{Gao-2025}, we can show that
\begin{equation*}		
\big\|\widehat{\lambda_{b}}-\widehat{\mathds{1}_{[N]}}\big\|_{\infty}=o(N).
\end{equation*}
Therefore, we have
\begin{equation*}
\big\|\widehat{f_{b,\gamma}}-\widehat{\mathds{1}_{[N]}}\big\|_{\infty}\leqslant \big\|\widehat{f_{b,\gamma}}-\widehat{\lambda_{b}}\big\|_{\infty}+
\big\|\widehat{\lambda_{b}}-\widehat{\mathds{1}_{[N]}} \big\|_{\infty}=o(N).
\end{equation*}
This completes the proof of Lemma \ref{pseudorandomness}.
\end{proof}

\begin{lemma}\label{Restriction estimate 1}
Let $ b\in [W] $ be subject to $ b\in Z(W)$. Let $\psi:[N]\to\mathbb{C}$ satisfy $|\psi|\leqslant\tau$,
where $\tau:[N]\rightarrow\mathbb{R}_{\geqslant0}$ is defined by
\begin{equation*}
\tau(n)=
\begin{cases}
\dfrac{2}{\sigma_{W}(b)}\gamma^{-1}t^{2-\gamma}, & \textrm{if\, $Wn+b=t^2$\,and $t\in\mathscr{N}_\gamma$}, \\
			0, & \textrm{otherwise}.
\end{cases}
\end{equation*}
Then, for any $\nu>4+\frac{26(1-\gamma)}{3\gamma-2}$, we have
\begin{equation*}
\int_{\mathbb{T}}\big|\widehat{\psi}(\alpha)\big|^\nu\mathrm{d}\alpha\ll N^{\nu-1}H^\nu,
\end{equation*}
where $H=\log\sqrt{WN}$.
\end{lemma}
\begin{proof}
One can follow the arguments exactly the same as those in \cite[Lemma 4.1]{Ren-Sun-Zhang-Zhang-2025} with
\cite[Eq. (1.7)]{Ren-Sun-Zhang-Zhang-2025} replaced by \cite[Lemma 10]{Akbal-Guloglu-2016}. Therefore, we omit the details herein.	
\end{proof}

It is easy to see that $0<\frac{26(1-\gamma)}{3\gamma-2}<1$ provided that $\gamma\in(\frac{28}{29},1)$, and then
we have the following result.
\begin{lemma}\label{Restriction estimate 2}
Suppose that $\gamma\in(\frac{28}{29},1)$. Then, for any $b\in[W]$ with $b\in Z(W)$, there exists $q$ subject to
$4<q<5$ such that
\begin{equation*}
\int_{\mathbb{T}}\big|\widehat{f_{b,\gamma}}(\alpha)\big|^{q}\mathrm{d}\alpha\ll N^{q-1}.
\end{equation*}
\end{lemma}
\begin{proof}
The conclusion follows directly by repeating the arguments in \cite[Proposition 5.1]{Ren-Zhang-Zhang-2024} with
\cite[Lemma 5.2]{Ren-Zhang-Zhang-2024} replaced by Lemma \ref{Restriction estimate 1}. Hence, we omit the details
herein.	
\end{proof}

\section{Mean value estimate}
Suppose that $P\subseteq[N]$ is an arithmetic progression with $|P|\geqslant N/w$, then there exist integers
$a,q$ such that $1\leqslant q\leqslant w$, $a\in[N]$ and $P=\{a+qt:1\leqslant t\leqslant|P|\}$. In this section, we shall establish the mean value result of $f_{b,\gamma}$ over the set $P$ for $\gamma\in(\frac{75}{82},1)$ and
$b\in[W]\cap Z(W)$, i.e., (\ref{mean-value}).
	
Set $L:=|P|,W':=Wq$ and $b':=Wa+b$. Since $q\leqslant w$ and $(b,W)=1$, we have
\begin{equation*}
1\leqslant(b',W')=(Wa+b,Wq)\leqslant(Wa+b,W)\cdot(Wa+b,q)=1.
\end{equation*}
By \cite[Proposition 4.2.1 and Proposition 4.2.2]{Ireland-Rosen-book} and the fact that $q\leqslant w$, we have
\begin{equation}\label{two sigma}
\sigma_{W}(b)=\sigma_{W'}(b').
\end{equation}
By the definition of $ f_{b,\gamma} $, we have
\begin{align*}
         \sum_{n\in P}f_{b,\gamma}(n)
= & \,\, \frac{2\varphi(W)}{W\sigma_{W}(b)}\gamma^{-1}\sum_{\substack{n\in P\\Wn+b=p^2\\p\in\mathscr{N}_\gamma}}
         p^{2-\gamma}(\log p)=\frac{2\varphi(W)}{W\sigma_W(b)}\gamma^{-1}
         \sum_{\substack{1\leqslant t \leqslant L \\W't+b'=p^2\\p\in \mathscr{N}_\gamma}}p^{2-\gamma}(\log p)
               \nonumber \\
= & \,\, \frac{2\varphi(W)}{W\sigma_{W}(b)}\gamma^{-1}\sum_{\substack{\sqrt{b'}<p\leqslant \sqrt{W'L+b'}\\ p^2\equiv b' \!\!\!\!\!\pmod{W'}\\ p\in\mathscr{N}_\gamma}}p^{2-\gamma}(\log p).
\end{align*}
It is easy to see that
\begin{equation*}
\frac{1}{W^{\prime}}\sum_{1\leqslant j\leqslant W'}e\bigg(j\cdot\frac{p^2-b'}{W'}\bigg)=
\begin{cases}
1, & \textrm{if\,\,$p^2 \equiv b' \!\!\!\!\!\pmod{W'}$}, \\
0, & \textrm{otherwise},
\end{cases}
\end{equation*}
which combined with Lemma \ref{trans} yields that
\begin{align}\label{f_b,gamma-mean}
		& \,\,  \sum_{n\in P}f_{b,\gamma}(n)
 = \frac{2\varphi(W)}{W\sigma_{W}(b)}\gamma^{-1}\sum_{\substack{\sqrt{b'}<p\leqslant \sqrt{W'L+b'}\\
          p\in\mathscr{N}_\gamma}}p^{2-\gamma}(\log p)\frac{1}{W'}\sum_{1\leqslant j\leqslant W'}
          e\bigg(j\cdot\frac{p^2-b'}{W'}\bigg)
                 \nonumber \\
 = & \,\, \frac{2\varphi(W)}{W\sigma_{W}(b)}\frac{1}{W'}\sum_{1\leqslant j\leqslant W'}
          e\bigg(-j\cdot\frac{b'}{W'}\bigg)\sum_{\substack{\sqrt{b'}<p\leqslant\sqrt{W'L+b'}\\
          p\in \mathscr{N}_\gamma}}\frac{1}{\gamma}p^{2-\gamma}(\log p)e\bigg(p^2\cdot\frac{j}{W'}\bigg)
                 \nonumber \\
 = & \,\, \frac{2\varphi(W)}{W\sigma_{W}(b)}\frac{1}{W'}\sum_{1\leqslant j\leqslant W'}
          e\bigg(-j\cdot\frac{b'}{W'}\bigg)\bigg(\sum_{\sqrt{b'}<p\leqslant\sqrt{W'L+b'}}p(\log p)
          e\bigg(p^2\cdot\frac{j}{W'}\bigg) +O(n_{0}^{1-\varepsilon}) \bigg)
                 \nonumber \\
 = & \,\, \frac{\varphi(W)}{W\sigma_{W}(b)}\sum_{\substack{\sqrt{b'}<p\leqslant\sqrt{W'L+b'}\\
          p^2\equiv b'\!\!\!\!\!\pmod{W'}}}2p\log p+O(n_{0}^{1-\varepsilon})
                 \nonumber \\
 = & \,\, \frac{\varphi(W)}{W\sigma_{W}(b)}\sum_{\substack{z\in[W']\\ z^2\equiv b' \!\!\!\!\!\pmod{W'}}}
          \sum_{\substack{\sqrt{b'}<p\leqslant\sqrt{W'L+b'}\\ p\equiv z \!\!\!\!\!\pmod{W'}}}
          2p\log p+O(n_{0}^{1-\varepsilon}),
		\end{align}
where we use the following estimate
\begin{equation*}
\sqrt{W'L+b'}=\sqrt{W(qL+a)+b}\leqslant \sqrt{WN+b}\ll n_{0}^{1/2}.
\end{equation*}
Next, we shall compute the inner sum on the right--hand side of (\ref{f_b,gamma-mean}). For this objective, we need to use the Siegel--Walfisz theorem.
\begin{lemma}[Siegel--Walfisz]\label{Siegel¨CWalfisz}
For $(\ell,q)=1$, let $\pi(x;q,\ell)$ denote the number of distinct primes not greater than $x$ which are in the arithmetic progression $\ell+kq$. Then, for any $A>0$, there holds
\begin{equation*}
\pi(x;q,\ell)=\frac{1}{\varphi(q)}\operatorname{li}(x)+O\big(xe^{-c\sqrt{\log x}}\big),
\end{equation*}
provided that $q\ll(\log x)^A$, where $\operatorname{li}(x)=\int_{2}^{x}\frac{\mathrm{d}t}{\log t}$, and the constants included in the $O$--symbol are independent of $q$.
\end{lemma}
\begin{proof}	
See Lemma 7.14 of Hua \cite{Hua-book}.
\end{proof}
Since $q\leqslant w$, then it follows that $W'\leqslant\log\sqrt{WN+W}$ for all sufficiently large $n_{0}$. Let
$X_1=\lfloor\sqrt{b'}\rfloor, X_2=\lfloor\sqrt{W'L+b'}\rfloor$ and $M=\sqrt{WN+W}$. By setting $g(n)=2n\log n$,
we get
\begin{align*}
         \sum_{\substack{X_1<p\leqslant X_2\\ p\equiv z \!\!\!\!\!\pmod{W'}}}2p\log p
= & \,\, \sum_{X_1<n\leqslant X_2}\big(\pi(n;W',z)-\pi(n-1;W',z)\big)g(n)
               \nonumber \\
= & \,\, \pi(X_2;W',z)g(X_2)-\pi(X_1; W',z)g(X_1+1)\\
			&\ \ \ \ +\sum_{X_1<n\leqslant X_2-1}\pi(n; W',z)(g(n)-g(n+1)).
\end{align*}
It follows from  mean value theorem that $g(n)-g(n+1)\ll\log X_{2}\leqslant\log M$, which combined with Lemma \ref{Siegel¨CWalfisz} yields that
\begin{align*}
         \sum_{\substack{X_1<p\leqslant X_2\\ p\equiv z \!\!\!\!\!\pmod{W'}}}2p\log p
= & \,\, \bigg(\frac{\operatorname{li}(X_2)}{\varphi(W')}+O\big(Me^{-c_1\sqrt{\log M}}\big)\bigg)g(X_2)
                   \nonumber \\
  & \,\, -\bigg(\frac{\operatorname{li}(X_1)}{\varphi(W')}+O\big(Me^{-c_1\sqrt{\log M}}\big)\bigg)g(X_{1}+1)
                   \nonumber \\
  & \,\, +\sum_{X_1<n\leqslant X_2-1}\bigg(\frac{\operatorname{li}(n)}{\varphi(W')}
         +O\big(Me^{-c_1\sqrt{\log M}}\big)\bigg)(g(n)-g(n+1))
                   \nonumber \\
= & \,\, \frac{1}{\varphi(W')}\bigg(\operatorname{li}(X_2)g(X_2)-\operatorname{li}(X_1)g(X_1+1)
                   \nonumber \\
  & \,\, +\sum_{X_1<n\leqslant X_2-1}\operatorname{li}(n)(g(n)-g(n+1))\bigg)+O\big(M^2e^{-c_2\sqrt{\log M}}\big)
                   \nonumber \\
= & \,\, \frac{1}{\varphi(W')}\sum_{X_1<n\leqslant X_2}g(n)\int_{n-1}^{n}\frac{\mathrm{d}t}{\log t}
         +O\big(M^{2}e^{-c_2\sqrt{\log M}}\big).
\end{align*}
By noting that $\log t=\log n+O(1/n)$ for $ n-1\leq t\leq n $, we derive that
\begin{align*}
         \sum_{\substack{X_1<p\leqslant X_2\\ p\equiv z \!\!\!\!\!\pmod{W'}}}2p\log p
= & \,\, \frac{1}{\varphi(W')}\sum_{X_1<n\leqslant X_2}2n+O\big(M^2e^{-c_2\sqrt{\log M}}\big)
               \nonumber \\
= & \,\, \frac{W'L}{\varphi(W')}+O\big(M^{2}e^{-c_2\sqrt{\log M}}\big).
\end{align*}
Since $q\leqslant w$, we get
\begin{equation}\label{equal}
\frac{W'}{\varphi(W')}=\frac{Wq}{\varphi(Wq)}=\frac{W}{\varphi(W)}.
\end{equation}
By (\ref{equal}) and $L\geqslant N/w$, we have
\begin{equation}\label{the error term}
\frac{M^2e^{-c_2\sqrt{\log M}}}{W'L/\varphi(W')}\leqslant
\frac{M^2e^{-c_2\sqrt{\log M}}}{W'N/(w\varphi(W'))}=
\frac{M^2e^{-c_2\sqrt{\log M}}w\varphi(W')}{W'N}=o(1).
\end{equation}
By (\ref{equal}) and (\ref{the error term}), we deduce that
\begin{equation}\label{inner sum}
\sum_{\substack{X_1<p\leqslant X_2\\ p\equiv z \!\!\!\!\!\pmod{W'}}}2p\log p=\dfrac{WL}{\varphi(W)}(1+o(1)).
\end{equation}
	Combining (\ref{two sigma}), (\ref{f_b,gamma-mean}) and (\ref{inner sum}) , we derive that
	\begin{equation}\label{mean-value}
		\begin{aligned}
			\dfrac{1}{|P|}\sum_{n\in P}f_{b,\gamma}(n)=1+o(1).
		\end{aligned}
	\end{equation}

\section{Proof of Theorem \ref{Theorem-1}}

In this section, we shall use Proposition \ref{transference principle} to prove Theorem \ref{Theorem-1}. Before presenting the proof, we need the following lemma about local condition of Waring--Goldbach problem.
\begin{lemma}\label{local-solution}
		For any integer \( n \equiv 5 \pmod{24} \), there exist integers $b_i$ with $1\leqslant b_i\leqslant W$ and $(W,b_i)=1$ for $i=1,\dots,5$, such that each $b_i$ is a quadratic residue modulo $W$ and
\begin{equation*}		
n\equiv b_1+b_2+b_3+b_4+b_5\!\!\!\pmod{W}.
\end{equation*}
\end{lemma}
\begin{proof}
See Lemma 4.3 of Bauer \cite{Bauer-2020}.
\end{proof}
	
\noindent
\textbf{Proof of Theorem \ref{Theorem-1}}. Let $n_{0}$ be a sufficiently large integer which satisfies
$n\equiv5\pmod{24}$. By Lemma \ref{local-solution}, we know that there exist integers $b_1,\dots,b_5\in[W]$
subject to
\begin{equation}\label{n_o-congru}
n_{0}\equiv b_1+b_2+b_3+b_4+b_5\!\!\!\pmod{W}
\end{equation}
and $b_{i}\in Z(W)$ for all $1\leqslant i \leqslant5$. It follows from Lemma \ref{Restriction estimate 2} that,
for $i=1,2,3,4,5$, the following estimates
\begin{equation*}
\big\|\widehat{f_{b_i,\gamma_i}}\big\|_{q}\leqslant KN^{1-1/q}
\end{equation*}
holds for some $q\in(4,5)$ and $K\geqslant1$. By (\ref{mean-value}), one has
\begin{equation*}
\mathbb{E}_{n\in P}f_{b_i,\gamma_i}(n)\geqslant 1/5 + \varepsilon
\end{equation*}
for any $\varepsilon\in(0,4/5)$. Moreover, by Lemma \ref{pseudorandomness}, it is easy to see that, for each arithmetic progression $P\subseteq[N]$ with $|P|\geqslant\eta N$, there holds
\begin{equation*}
\big\|\widehat{f_{b_i,\gamma_i}}-\widehat{\mathds{1}_{[N]}}\big\|_{\infty}\leqslant\eta N,
\end{equation*}
where $\eta>0$ is a sufficiently small constant which will be determined later.
Based on the above arguments, it follows from Proposition \ref{transference principle} that
\begin{equation*}
f_{b_1,\gamma_1}\ast\cdots\ast f_{b_5,\gamma_5}(n)\geqslant
\big(c(\varepsilon)-\eta\cdot C(\varepsilon,q,K)\big)N^{4}
\end{equation*}
holds for all $n\in[N/2,N]$, where $C(\varepsilon,q,K)>0 $ is a constant depending only on $\varepsilon,q$ and
$K$. let $\eta$ be sufficiently small such that
\begin{equation*}
c(\varepsilon)-\eta\cdot C(\varepsilon,q,K)>0,
\end{equation*}
and thus
\begin{equation*}
f_{b_1,\gamma_1}\ast\cdots\ast f_{b_5,\gamma_5}(n)>0
\end{equation*}
for all $n\in[N/2,N]$, which means
\begin{equation*}
\sum_{n=n_1+n_2+n_3+n_4+n_5}f_{b_1,\gamma_1}(n_1)f_{b_2,\gamma_2}(n_2)f_{b_3,\gamma_3}(n_3)f_{b_4,\gamma_4}(n_4)
f_{b_5,\gamma_5}(n_5)>0
\end{equation*}
for all $n\in[N/2,N]$. In other words, for all $n\in[N/2,N]$, there exist integers $n_1,n_2,n_3,n_4,n_5$
such that $n=n_1+n_2+n_3+n_4+n_5$ with $Wn_i+b_i=p_i^2$, where $p_i\in\mathscr{N}_{\gamma_i}\cap\mathcal{P}$. Therefore, we deduce that
\begin{equation*}
Wn+b_1+\dots+b_5=(Wn_1+b_1)+\dots+(Wn_5+b_5)=p_1^2+\dots+p_5^2\in P_1+\dots+P_5,
\end{equation*}
where $P_i:=\{p^2:p\in\mathscr{N}_{\gamma_i}\cap\mathcal{P}\} $.
Let $n'=(n_0-b_1-\dots-b_5)/W$. It follows from (\ref{n_o-congru}) and $N=\lfloor n_0/W\rfloor+1$ that
 $n'\in\mathbb{Z}^+$ and $N/2<n'<N$. Hence, we deduce that
\begin{equation*}
  n_0=Wn'+b_1+\dots+b_5\in P_1+\dots+P_5.
\end{equation*}
This completes the proof of Theorem \ref{Theorem-1}.

\section*{Acknowledgement}

The authors would like to appreciate the referee for his/her patience in refereeing this paper.
This work is supported by Beijing Natural Science Foundation (Grant No. 1242003), and
the National Natural Science Foundation of China (Grant Nos. 12471009, 12301006, 11901566, 12001047).


\begin{thebibliography}{99}

\bibitem{Akbal-Guloglu-2016}Y. Akbal, A. M. G\"{u}lo\u{g}lu, \textit{Waring's problem with Piatetski--Shapiro
                     numbers}, Mathematika, \textbf{62} (2016), no. 2, 524--550.

\bibitem{Balog-Friedlander-1992}A. Balog, J. Friedlander, \textit{A hybrid of theorems of Vinogradov and
                     Piatetski--Shapiro}, Pacific J. Math., \textbf{156} (1992), no. 1, 45--62.

\bibitem{Bauer-2020}C. Bauer, \textit{An L--function free proof of Hua's theorem on sums of five prime squares},
                    Studia Sci. Math. Hungar, \textbf{57} (2020), no. 1, 1--39.

\bibitem{Chow-2018}S. Chow, \textit{Roth--Waring--Goldbach},
                     Int. Math. Res. Not. IMRN, \textbf{2018} (2018), no. 8, 2341--2374.

\bibitem{Cui-2004}Z. Cui, \textit{Hua's theorem with the primes in Shapiro prime sets},
                             Acta Math. Hungar.,  \textbf{104} (2004), no. 4, 323--329.

\bibitem{Gao-2025}M. Gao, \textit{A density version of Waring--Goldbach problem},
                   Int. J. Number Theory, \textbf{21} (2025), no. 6, 1417--1436.

\bibitem{Heath-Brown-1983}D. R. Heath--Brown, \textit{The Pjatecki\u{\i}--\u{S}apiro prime number theorem},
                                  J. Number Theory, \textbf{16} (1983), no. 2, 242--266.

\bibitem{Hua-1938}L. K. Hua, \textit{Some results in the additive prime--number theory},
                          Quart. J. Math. Oxford Ser. (2), \textbf{9} (1938), no. 1, 68--80.

\bibitem{Hua-book}L. K. Hua, \textit{Additive Theory of Prime Numbers}, American Mathematical Society,
                       Providence, Rhode Island, 1965.

\bibitem{Ireland-Rosen-book}K. Ireland, M. Rosen, \textit{A classical introduction to modern number theory},
                  Graduate Texts in Mathematics 84, Springer--Verlag, New York, 1990.

\bibitem{Jia-1995}C.-H. Jia, \textit{On the Piatetski--Shapiro--Vinogradov theorem}, Acta Arith.,
                                            \textbf{73} (1995), no. 1, 1--28.

\bibitem{Kolesnik-1967}G. Kolesnik, \textit{The distribution of primes in sequences of the form $[n^c]$},
                          Mat. Zametkki, \textbf{2} (1967), 117--128.

\bibitem{Kolesnik-1985}G. Kolesnik, \textit{Primes of the form $[n^c]$},
                          Pacific J. Math., \textbf{118} (1985), no. 2, 437--447.

\bibitem{Kumchev-1997}A. Kumchev, \textit{On the Piatetski--Shapiro--Vinogradov theorem},
                                          J. Th\'{e}or. Nombres Bordeaux, \textbf{9} (1997), no. 1, 11--23.

\bibitem{Leitmann-1975}D. Leitmann, D. Wolke, \textit{Primzahlen der Gestalt $[f(n)]$},
                           Math. Z., \textbf{145} (1975), no. 1, 81--92.

\bibitem{Leitmann-1980}D. Leitmann, \textit{Absch\"{a}tzung trigonometrischer Summen},
                            J. Reine Angew. Math., \textbf{317} (1980), 209--219.

\bibitem{Li-Zhang-2018}J. Li, M. Zhang, \textit{Hua's theorem with the primes in Piatetski--Shapiro prime sets},
                            Int. J. Number Theory, \textbf{14} (2018), no. 1, 193--220.

\bibitem{Liu-Rivat-1992}H. Q. Liu, J. Rivat, \textit{On the Pjatecki\u{\i}--\u{S}apiro prime number theorem},
                           Bull. London Math. Soc., \textbf{24} (1992), no. 2, 143--147.

\bibitem{Piatetski-Shapiro-1953}I. I. Pyatecki\u{\i}--\u{S}apiro, \textit{On the distribution of prime numbers in
                  sequences of the form $[f(n)]$},  Mat. Sbornik N.S., \textbf{33(75)} (1953), 559--566.v

\bibitem{Ren-Sun-Zhang-Zhang-2025}X. M. Ren, Y. C. Sun, Q. Q. Zhang, R. Zhang, \textit{Roth--type theorem for nonlinear equations in Piatetski--Shapiro primes}, Int. J. Number Theory, \textbf{21} (2025), no. 4, 887--902.

\bibitem{Ren-Zhang-Zhang-2024}X. M. Ren, Q. Q. Zhang, R. Zhang, \textit{Roth--type theorem for quadratic system
              in Piatetski--Shapiro primes}, J. Number Theory, \textbf{257} (2024), 1--23.

\bibitem{Rivat-1992}J. Rivat, \textit{Autour d'un th\'{e}or\'{e}me de Pjatecki\u{\i}--\u{S}apiro: Nombres premiers
                 dans la suite $[n^c]$}, Th\'{e}se de Doctorat, Universit\'{e} Paris--Sud, 1992.

\bibitem{Rivat-Sargos-2001}J. Rivat, P. Sargos, \textit{Nombres premiers de la forme $\lfloor n^c \rfloor$},
                      Canad. J. Math., \textbf{53} (2001), no. 2, 414--433.

\bibitem{Rivat-Wu-2001}J. Rivat, J. Wu, \textit{Prime numbers of the form $[n^c]$},
                            Glasg. Math. J., \textbf{43} (2001), no. 2, 237--254.

\bibitem{Salmensuu-2020}J. Salmensuu, \textit{On the Waring--Goldbach problem with almost equal summands},
               Mathematika, \textbf{66} (2020), no. 2, 255--296.

\bibitem{Sun-Du-Pan-2025}Y. C. Sun, S. S. Du, H. Pan, Hao, \textit{Vinogradov's theorem with
                 Piatetski--Shapiro primes}, Int. Math. Res. Not. IMRN, 2025, no. 15, Paper No. rnaf125, 30 pp.

\bibitem{Vinogradov-1937}I. M. Vinogradov, \textit{Representation of an odd number as the sum of three primes},
                      Dokl. Akad. Nauk. SSSR, \textbf{15} (1937), 169--172.

\bibitem{Zhai-1998}W. G. Zhai, \textit{The Waring--Goldbach problem in thin sets of primes},
                    Acta Math. Sinica (Chinese Ser.), \textbf{41} (1998), no. 3, 595--608.

\bibitem{Zhang-Zhai-2005}D. Y. Zhang, W. G. Zhai, \textit{The Waring--Goldbach problem in thin sets of
             primes (II)}, Acta Math. Sinica (Chinese Ser.), \textbf{48} (2005), no. 4, 809--816.



\end{thebibliography}
\end{document}